\documentclass[reqno]{amsart}
\usepackage{amssymb,amsmath,hyperref}
\usepackage{amsrefs, float}
\usepackage[foot]{amsaddr}
\usepackage{bbold,stackrel, multicol}
\usepackage{mathrsfs}
 
\newtheorem{thm}{Theorem}

\newtheorem{cor}[thm]{Corollary}
\newtheorem{defi}[thm]{Definition}

\newtheorem{rem}[thm]{Remark}
\newtheorem{nota}[thm]{Notation}

\newtheorem*{tempo*}{Template}

\newcommand\be{\begin{equation}}
\newcommand\ee{\end{equation}} 

\usepackage{amsmath,amsfonts} 

\usepackage[applemac]{inputenc}

\def\bdefi{\begin{defi}\rm}
\def\edefi{\end{defi}}
\def\bnota{\begin{nota}\rm}
\def\enota{\end{nota}}

\def\ZFC{\textup{\textsf{ZFC}}}

\def\({\textup{(}}
\def\){\textup{)}}

\def\bye{\end{document}}
\def\N{{\mathbb  N}}
\def\Q{{\mathbb  Q}}
\def\R{{\mathbb  R}}

\def\SS{\textup{\textsf{S}}}

\def\osc{\textup{\textsf{osc}}}

\def\di{\rightarrow}

\def\asa{\leftrightarrow}

\def\u{\textup{\textsf{u}}}

\def\fin{\textup{\textsf{fin}}}

\def\eps{\varepsilon}

\newcommand{\F}{{\bf F}}

\usepackage{graphicx}
\usepackage{tikz}
\usetikzlibrary{matrix, shapes.misc}
\usepackage{comment,tikz-cd}

\numberwithin{equation}{section}
\numberwithin{thm}{section}

\begin{document}
\title[]{On the computational properties of the Baire category theorem}
\author{Sam Sanders}
\address{Department of Philosophy II, RUB Bochum, Germany}
\email{sasander@me.com}
\begin{abstract}
%
Computability theory is a discipline in the intersection of computer science and mathematical logic where the fundamental question is:
\begin{center}
\emph{given two mathematical objects $X$ and $ Y$, does $X$ compute $Y$ {in principle}?}
\end{center} 
In case $X$ and $ Y$ are real numbers, Turing's famous `machine' model provides the standard interpretation of `computation' for this question.
To formalise computation involving (total) abstract objects, Kleene introduced his S1-S9 computation schemes.
In turn, Dag Normann and the author have  introduced a version of the lambda calculus involving fixed point operators that \emph{exactly} captures S1-S9 \emph{and} accommodates partial objects. 
In this paper, we use this new model to develop the computability theory of various well-known theorems due to Baire and Volterra and related results; these theorems only require basic mathematical notions like continuity, open sets, and density. 
We show that these theorems due to Baire and Volterra are \emph{computationally equivalent} from the point of view of our new model, sometimes working in rather tame fragments of G\"odel's $T$.
\end{abstract}


\maketitle
\thispagestyle{empty}


\section{Introduction}\label{intro}
\subsection{Motivation and overview}\label{mintro}
Computability theory is a discipline in the intersection of theoretical computer science and mathematical logic where the fundamental question is:
\begin{center}
\emph{given two mathematical objects $X $ and $ Y$, does $X$ compute $Y$ {in principle}?}
\end{center} 
In case $X $ and $Y$ are real numbers, Turing's famous `machine' model (\cite{tur37}) is the standard approach to this question, i.e.\ `computation' is interpreted in the sense of Turing machines.  
To formalise computation involving (total) abstract objects, like functions on the real numbers or well-orderings of the reals, Kleene introduced his S1-S9 computation schemes (\cites{kleeneS1S9, longmann}).

\smallskip

In turn, Dag Normann and the author have recently introduced (\cite{dagsamXIII}) a version of the lambda calculus involving fixed point operators that exactly captures S1-S9 and accommodates partial objects. 
In this paper, we use this new model to develop the computability theory of various theorems due to Baire and Volterra and related results.  
We show that these theorems are \emph{computationally equivalent} from the point of view of our new model, as laid down by Definition \ref{spec} below. 
We stress that these theorems due to Baire and Volterra only require basic mathematical notions, like continuity, open sets, and density. 
All non-basic definitions and some background may be found in Sections \ref{vintro} and \ref{kelim}.

\medskip
\noindent
In more detail, consider the following three theorems due to Baire and Volterra from the 19th century (\cites{beren2, volaarde2}).
The notion of \emph{Baire 1 function} is also due to Baire (\cite{beren2}) and means that the function is the limit of a sequence of continuous functions.  
\begin{itemize}
\item (Baire category theorem) Let $(O_{n})_{n\in \N}$ be a sequence of dense and open sets of reals.  Then the intersection $\cap_{n\in \N}O_{n}$ is non-empty (and dense).
\item (Volterra) There is no function that is continuous on the rationals and discontinuous on the irrationals.  
\item  (Baire) For $f:[0,1]\di \R$ in Baire 1, the set of continuity points is non-empty (and dense).  
\end{itemize}  
The first item is well-studied in other computational approaches to mathematics like \emph{computable analysis} (\cites{brakke, brakke2}) or \emph{constructive mathematics} (\cite{bish1}*{p.\ 84} and \cite{nemoto1}).  
The second item is a mainstay of popular mathematics (\cites{shola, kinkvol}), and usually proved via the first item (\cite{dendunne}).  More detailed (historical) background may be found in Section~\ref{vintro}.  

\medskip
\noindent
The above theorems yield the following natural computational operations.  A functional performing the operation in item \eqref{don1} is called a \emph{Baire realiser} (\cite{dagsamVII}).
\begin{enumerate}
\renewcommand{\theenumi}{\alph{enumi}}
\item On input a sequence of dense and open sets of reals $(O_{n})_{n\in \N}$, output a real $x \in \cap_{n\in \N}O_{n}$.\label{don1}
\item On input a Baire 1 function from reals to reals, output a rational where it is discontinuous, or an irrational where it is continuous.\label{don2}
\item On input a Baire 1 function from reals to reals, output a point of contuity.\label{don3}
\end{enumerate}  
We show in Section \ref{main3} that the operations in items \eqref{don1}-\eqref{don3}, as well as a number of related ones, are \emph{computationally equivalent} as defined in Definition \ref{spec} below.
We emphasise that each of the items \eqref{don1}-\eqref{don3} should be viewed as a specification for a class of (possibly partial) functionals that produce certain outputs given certain inputs.
The computational equivalence between two specifications $\textsf{\textup{(A)}}$ and $\textsf{\textup{(B)}}$ is then defined as follows.  
\begin{defi}[Computational equivalence]\label{spec} 
Let $\textsf{\textup{(A)}}$ and $\textsf{\textup{(B)}}$ be specifications for classes of functionals.  
We say that $\textsf{\textup{(A)}}$ and $\textsf{\textup{(B)}}$ are \emph{computationally equivalent} if there are algorithms with index $e, d\in \N$ such that:
\begin{itemize}
\item for any functional $\Gamma$ satisfying \textsf{\textup{(A)}}, the functional $\{e\}(\Gamma)$ satisfies \textsf{\textup{(B)}},
\item for any functional $\Lambda$ satisfying \textsf{\textup{(B)}}, the functional $\{d\}(\Lambda)$ satisfies \textsf{\textup{(A)}}. 
\end{itemize}
\edefi
In most cases below, we can replace Kleene's S1-S9 or our model from \cite{dagsamXIII} by rather tame fragments of G\"odel's $T$ and Feferman's $\mu^{2}$ from Section \ref{lll}.  
We note that our specifications are generally restricted to functionals that map $\R\di \R$-functions to $\R\di \R$-functions, or constructs of similar rank limitations.

\medskip

Next, to study of continuous functions in computability theory (or constructive mathematics), 
such functions are usually given together with a \emph{modulus\footnote{Intuitively speaking, a modulus of continuity is a functional that witnesses the usual `epsilon-delta' definition of continuity.  In this way, a modulus of continuity essentially outputs a suitable `delta' on input any $\eps>0$ and $x$ in the domain, as in the usual epsilon-delta definition.} of continuity}, i.e.\ a constructive enrichment providing computational information about continuity.  
As discussed in detail in Section \ref{prebaire}, our study of Baire 1 functions in computability theory also makes use of some kind of constructive enrichment, 
namely that Baire 1 functions are given together with their \emph{oscillation function} (see Def.\ \ref{oscf}), going back to Riemann and Hankel (\cites{hankelwoot, rieal}).

\medskip

Finally, we note that by \cite{dagsamVII}*{\S6}, the operations in items \eqref{don1}-\eqref{don3} are \emph{hard\footnote{The functional $\SS_{k}^{2}$ from Section \ref{lll} can decide $\Pi_{k}^{1}$-formulas, but the operation in item \eqref{don1} is not computable in $\SS_{k}^{2}$ (or their union).  Kleene's $\exists^{3}$ from Section \ref{lll} computes the operation in item \eqref{don1}, but the former also yields full second-order arithmetic, as does the union of all $\SS_{k}^{2}$ functionals.\label{klank}} to compute} relative to the usual hierarchy based on comprehension.  
An explanation for this phenomenon may be found in Section \ref{prelim}.

\subsection{Volterra's early work and related results}\label{vintro}
We introduce Volterra's early work from \cite{volaarde2} as it pertains to this paper.  We let $\R$ be the set of real numbers while $\Q$ is the set of rational numbers. 

\medskip

First of all, the Riemann integral was groundbreaking in its day for a number of reasons, including its ability to integrate functions with infinitely many points of discontinuity, as shown by Riemann himself (\cite{riehabi}). 
A natural question is then `how discontinuous' a Riemann integrable function can be.  In this context, Thomae introduced the folllowing function $T:\R\di\R$ around 1875 in \cite{thomeke}*{p.\ 14, \S20}):
\be\label{thomae}\tag{\textup{\textsf{T}}}
T(x):=
\begin{cases} 
0 & \textup{if } x\in \R\setminus\Q\\
\frac{1}{q} & \textup{if $x=\frac{p}{q}$ and $p, q$ are co-prime} 
\end{cases}.
\ee
Thomae's function $T$ is Riemann integrable on any interval, but has a dense set of points of discontinuity, namely $\Q$, and a dense set of points of continuity, namely $\R\setminus \Q$. 

\medskip

The perceptive student, upon seeing Thomae's function as in \eqref{thomae}, will ask for a function continuous at each rational point and discontinuous at each irrational one.
Such a function cannot exist, as is generally proved using the Baire category theorem.  
However, Volterra in \cite{volaarde2} already established this negative result about twenty years before the publication of the Baire category theorem, using a rather elementary argument.

\medskip

Secondly, as to the content of Volterra's \cite{volaarde2}, we find the following theorem on the first page, where a function is \emph{pointwise discontinuous} if it has a dense set of continuity points.
\begin{thm}[Volterra, 1881]\label{VOL}
There do not exist pointwise discontinuous functions defined on an interval for which the continuity points of one are the discontinuity points of the other, and vice versa.
\end{thm}
Volterra then states two corollaries, of which the following is perhaps well-known in `popular mathematics' and constitutes the aforementioned negative result. 
\begin{cor}[Volterra, 1881]\label{VOLcor}
There is no $\R\di\R$ function that is continuous on $\Q$ and discontinuous on $\R\setminus\Q$. 
\end{cor}
Thirdly, we shall study Volterra's theorem and corollary restricted to Baire 1 functions (see Section \ref{cdef}). 
The latter kind of functions are automatically `pointwise discontinuous' in the sense of Volterra.

\medskip

Fourth,  Volterra's results from \cite{volaarde2} are generalised in \cites{volterraplus,gaud}.  
The following theorem is immediate from these generalisations. 
\begin{thm}\label{dorki}
For any countable dense set $D\subset [0,1]$ and $f:[0,1]\di \R$, either $f$ is discontinuous at some point in $D$ or continuous at some point in $\R\setminus D$. 
\end{thm}
Perhaps surprisingly, this generalisation is computationally equivalent to the original (both restricted to Baire 1 functions).  

\medskip

Fifth, the Baire category theorem for the real line was first proved by Osgood (\cite{fosgood}) and later by Baire (\cite{beren2}) in a more general setting.
\begin{thm}[Baire category theorem]
If $ (O_n)_{n \in \N}$ is a sequence of dense open sets of reals, then 
$ \bigcap_{n \in\N } O_n$ is non-empty.
\end{thm}
Based on this theorem, Dag Normann and the author have introduced the following in \cite{dagsamVII}.  
\bdefi\label{pip}
A Baire realiser $\xi: \big(\N\di (\R\di \N)\big)\di \R$ is a functional such that if $ (O_n)_{n \in \N}$ is a sequence of dense open sets of reals, $\xi(\lambda n. O_{n})\in \cap_{n\in \N}O_{n}$.
\edefi
We shall obtain a number of computational equivalences (Definition \ref{spec}) for Baire realisers in Section \ref{main3}. 
To this end, we will need some preliminaries and definitions, as in Section \ref{kelim}.

\medskip

Finally, in light of our definition of open set in Definition \ref{char}, our study of the Baire category theorem (and the same for \cite{dagsamVII})
is based on a most general definition of open set, i.e.\ no additional (computational) information is given.  Besides the intrinsic interest of such an investigation,
there is a deeper reason, as discussed in Section \ref{dichtbij}.  In a nutshell, in the study of Baire 1 functions, one readily encounters open sets `in the wild'
that do not come with any additional (computational) information.  In fact, finding such additional (computational) information turns out to be at least as hard 
as finding a Baire realiser.

%
%


\subsection{Preliminaries and definitions}\label{kelim}
We briefly introduce Kleene's \emph{higher-order computability theory} in Section~\ref{prelim}.
We introduce some essential axioms (Section~\ref{lll}) and definitions (Section~\ref{cdef}).  A full introduction may be found in e.g.\ \cite{dagsamX}*{\S2}.
Since Kleene's computability theory borrows heavily from type theory, we shall often use common notations from the latter; for instance, the natural numbers are type $0$ objects, denoted $n^{0}$ or $n\in \N$.  
Similarly, elements of Baire space are type $1$ objects, denoted $f\in \N^{\N}$ or $f^{1}$.  Mappings from Baire space $\N^{\N}$ to $\N$ are denoted $Y:\N^{\N}\di \N$ or $Y^{2}$. 
An overview of this kind of notations is in Section \ref{appendisch} and the general literature (see in particular \cite{longmann}). 

\subsubsection{Kleene's computability theory}\label{prelim}
Our main results are in computability theory and we make our notion of `computability' precise as follows.  
\begin{enumerate}
\item[(I)] We adopt $\ZFC$, i.e.\ Zermelo-Fraenkel set theory with the Axiom of Choice, as the official metatheory for all results, unless explicitly stated otherwise.
\item[(II)] We adopt Kleene's notion of \emph{higher-order computation} as given by his nine clauses S1-S9 (see \cite{longmann}*{Ch.\ 5} or \cite{kleeneS1S9}) as our official notion of `computable' involving total objects.
\end{enumerate}
We mention that S1-S8 are rather basic and merely introduce a kind of higher-order primitive recursion with higher-order parameters. 
The real power comes from S9, which essentially hard-codes the \emph{recursion theorem} for S1-S9-computability in an ad hoc way.  
By contrast, the recursion theorem for Turing machines is derived from first principles in \cite{zweer}.

\medskip

On a historical note, it is part of the folklore of computability theory that many have tried (and failed) to formulate models of computation for objects of all finite type and in which one derives the recursion theorem in a natural way.  For this reason, Kleene ultimately introduced S1-S9, which 
were initially criticised for their aforementioned ad hoc nature, but eventually received general acceptance.  

\medskip

Now, Dag Normann and the author have introduced a new computational model based on the lambda calculus in \cite{dagsamXIII} with the following properties:
\begin{itemize}
\item S1-S8 is included while the `ad hoc' scheme S9 is replaced by more natural (least) fixed point operators,
\item the new model exactly captures S1-S9 computability for total objects,
\item the new model accommodates `computing with partial objects',
\item the new model is more modular than S1-S9 in that sub-models are readily obtained by leaving out certain fixed point operators.
\end{itemize}
We refer to \cites{longmann, dagsamXIII} for a thorough overview of higher-order computability theory.
We do mention the distinction between `normal' and `non-normal' functionals  based on the following definition from \cite{longmann}*{\S5.4}. 
We only make use of $\exists^{n}$ for $n=2,3$, as defined in Section \ref{lll}.
\bdefi\label{norma}
For $n\geq 2$, a functional of type $n$ is called \emph{normal} if it computes Kleene's $\exists^{n}$ following S1-S9, and \emph{non-normal} otherwise.  
\edefi
\noindent
It is a historical fact that higher-order computability theory, based on Kleene's S1-S9 schemes, has focused primarily on the world of \emph{normal} functionals; this opinion can be found \cite{longmann}*{\S5.4}.  
Nonetheless, we have previously studied the computational properties of new \emph{non-normal} functionals, namely those that compute the objects claimed to exist by:
\begin{itemize}
\item covering theorems due to Heine-Borel, Vitali, and Lindel\"of (\cites{dagsam, dagsamII, dagsamVI}),
\item the Baire category theorem (\cite{dagsamVII}),
\item local-global principles like \emph{Pincherle's theorem} (\cite{dagsamV}),
\item weak fragments of the Axiom of (countable) Choice (\cite{dagsamIX}),
\item the uncountability of $\R$ and the Bolzano-Weierstrass theorem for countable sets in Cantor space (\cites{dagsamX, dagsamXI}),
\item the Jordan decomposition theorem and related results (\cites{dagsamXII, dagsamXIII}).
\end{itemize}
In this paper, we greatly extend the study of the Baire category theorem mentioned in the second item; the operations sketched in Section \ref{mintro} are all non-normal, in that they do not compute Kleene's $\exists^{2}$ from Section \ref{lll}.

\medskip

Finally, we have obtained many computational equivalences in \cite{dagsamXIII, samwollic22}, mostly related to the Jordan decomposition theorem and the uncountability of $\R$.
With considerable effort, we even obtained terms of G\"odel's $T$ witnessing (some of) these equivalences (see e.g.\ \cite{dagsamXIII}*{Thm.~4.14}).  
Many of the below equivalences go through in tame fragments G\"odel's $T$ extended with Feferman's $\mu^{2}$ from Section~\ref{lll}.

\subsubsection{Some comprehension functionals}\label{lll}
In Turing-style computability theory, computational hardness is measured in terms of where the oracle set fits in the well-known comprehension hierarchy.  
For this reason, we introduce some axioms and functionals related to \emph{higher-order comprehension} in this section.
We are mostly dealing with \emph{conventional} comprehension here, i.e.\ only parameters over $\N$ and $\N^{\N}$ are allowed in formula classes like $\Pi_{k}^{1}$ and $\Sigma_{k}^{1}$.  

\medskip

First of all, the functional $\varphi^{2}$, also called \emph{Kleene's quantifier $\exists^{2}$}, as in $(\exists^{2})$ is clearly discontinuous at $f=11\dots$; in fact, $\exists^{2}$ is (computationally) equivalent to the existence of $F:\R\di\R$ such that $F(x)=1$ if $x>_{\R}0$, and $0$ otherwise via Grilliot's trick (see \cite{kohlenbach2}*{\S3}).
\be\label{muk}\tag{$\exists^{2}$}
(\exists \varphi^{2}\leq_{2}1)(\forall f^{1})\big[(\exists n)(f(n)=0) \asa \varphi(f)=0    \big]. 
\ee
Related to $(\exists^{2})$, the functional $\mu^{2}$ in $(\mu^{2})$ is called \emph{Feferman's $\mu$} (\cite{avi2}).
\begin{align}\label{mu}\tag{$\mu^{2}$}
(\exists \mu^{2})(\forall f^{1})\big(\big[ (\exists n)(f(n)=0) \di [f(\mu(f))=0&\wedge (\forall i<\mu(f))(f(i)\ne 0) \big]\\
& \wedge [ (\forall n)(f(n)\ne0)\di   \mu(f)=0] \big). \notag
\end{align}
We have $(\exists^{2})\asa (\mu^{2})$ over Kohlenbach's base theory (\cite{kohlenbach2}), while $\exists^{2}$ and $\mu^{2}$ are also computationally equivalent.  
Hilbert and Bernays formalise considerable swaths of mathematics using only $\mu^{2}$ in \cite{hillebilly2}*{Supplement IV}.

\medskip
\noindent
Secondly, the functional $\SS^{2}$ in $(\SS^{2})$ is called \emph{the Suslin functional} (\cite{kohlenbach2}).
\be\tag{$\SS^{2}$}
(\exists\SS^{2}\leq_{2}1)(\forall f^{1})\big[  (\exists g^{1})(\forall n^{0})(f(\overline{g}n)=0)\asa \SS(f)=0  \big].
\ee
By definition, the Suslin functional $\SS^{2}$ can decide whether a $\Sigma_{1}^{1}$-formula as in the left-hand side of $(\SS^{2})$ is true or false.   
We similarly define the functional $\SS_{k}^{2}$ which decides the truth or falsity of $\Sigma_{k}^{1}$-formulas.
%
We note that the Feferman-Sieg operators $\nu_{n}$ from \cite{boekskeopendoen}*{p.\ 129} are essentially $\SS_{n}^{2}$ strengthened to return a witness (if existant) to the $\Sigma_{n}^{1}$-formula at hand.  

\medskip

\noindent
Thirdly, the functional $E^{3}$ clearly computes $\exists^{2}$ and $\SS_{k}^{2}$ for any $k\in \N$:
\be\tag{$\exists^{3}$}
(\exists E^{3}\leq_{3}1)(\forall Y^{2})\big[  (\exists f^{1})(Y(f)=0)\asa E(Y)=0  \big].
\ee
The functional from $(\exists^{3})$ is also called \emph{Kleene's quantifier $\exists^{3}$}, and we use the same -by now obvious- convention for other functionals.  
Hilbert and Bernays introduce a functional $\nu^{3}$ in \cite{hillebilly2}*{Supplement IV}, which is similar to $\exists^{3}$.

\medskip

In conclusion, the operations sketched in Section \ref{mintro} are computable in $\exists^{3}$ but not in any $\SS_{k}^{2}$, as noted in Footnote \ref{klank} and which immediately follows from \cite{dagsamVII}*{\S6}.  
Many non-normal functionals exhibit the same `computational hardness' and we merely view this as support for the development of a separate scale for classifying non-normal functionals.    

\subsubsection{Some definitions}\label{cdef}
We introduce some definitions needed in the below, mostly stemming from mainstream mathematics.
We note that subsets of $\R$ are given by their characteristic functions (Definition \ref{char}), where the latter are common in measure and probability theory.
In this paper, `continuity' refers to the usual `epsilon-delta' definition, well-known from the literature. 

\medskip
\noindent
Zeroth of all, we make use the usual definition of (open) set, where $B(x, r)$ is the open ball with radius $r>0$ centred at $x\in \R$.
\bdefi[Set]\label{char}~
\begin{itemize}
\item Subsets $A\subset \R$ are given by its characteristic function $F_{A}:\R\di \{0,1\}$, i.e.\ we write $x\in A$ for $ F_{A}(x)=1$ for all $x\in \R$.
\item A subset $O\subset \R$ is \emph{open} in case $x\in O$ implies that there is $k\in \N$ such that $B(x, \frac{1}{2^{k}})\subset O$.
\item A subset $C\subset \R$ is \emph{closed} if the complement $\R\setminus C$ is open. 
\end{itemize}
\edefi
\noindent
The reader will find more motivation for our definition of open set in Section~\ref{dichtbij}.

\medskip

\noindent
First of all, we study the following continuity notions, in part due to Baire (\cite{beren2}). 
\bdefi[Weak continuity]\label{flung} 
For $f:[0,1]\di \R$, we have the following:
\begin{itemize}
\item $f$ is \emph{upper semi-continuous} at $x_{0}\in [0,1]$ if $f(x_{0})\geq_{\R}\lim\sup_{x\di x_{0}} f(x)$,
\item $f$ is \emph{lower semi-continuous} at $x_{0}\in [0,1]$ if $f(x_{0})\leq_{\R}\lim\inf_{x\di x_{0}} f(x)$,
\item $f$ is \emph{quasi-continuous} at $x_{0}\in [0, 1]$ if for $ \epsilon > 0$ and an open neighbourhood $U$ of $x_{0}$, 
there is a non-empty open ${ G\subset U}$ with $(\forall x\in G) (|f(x_{0})-f(x)|<\eps)$.
\item $f$ is \emph{regulated} if for every $x_{0}$ in the domain, the `left' and `right' limit $f(x_{0}-)=\lim_{x\di x_{0}-}f(x)$ and $f(x_{0}+)=\lim_{x\di x_{0}+}f(x)$ exist.  
\end{itemize}
In case the weak continuity notion is satisfied at each point of the domain, the associated function satisfies the italicised notion. 
\edefi
Secondly, Baire introduces the hierarchy of Baire classes in \cite{beren2}.  We shall only need the first Baire class, defined as follows.
\bdefi[Baire 1 function]
A function $f:\R\di \R$ is \emph{Baire 1} if there exists a sequence $(f_{n})_{n\in \N}$ of continuous $\R\di \R$-functions such that $f(x)=\lim_{n\di\infty}f_{n}(x)$ for all $x\in \R$
\edefi
Thirdly, the following sets are often crucial in proofs relating to discontinuous functions, as can be observed in e.g.\ \cite{voordedorst}*{Thm.\ 0.36}.
\bdefi
The sets $C_{f}$ and $D_{f}$ respectively gather the points where $f:\R\di \R$ is continuous and discontinuous.
\edefi
One problem with the sets $C_{f}, D_{f}$ is that the definition of continuity involves quantifiers over $\R$.  
In general, deciding whether a given $\R\di \R$-function is continuous at a given real, is as hard as $\exists^{3}$ from Section \ref{lll}.
For these reasons, the sets $C_{f}, D_{f}$ do exist, but are not available as inputs for algorithms in general.  
A solution is discussed in Section \ref{prebaire}.

\medskip

Fourth, we introduce two notions to be found already in e.g.\ the work of Volterra, Smith, and Hankel (\cites{hankelwoot, snutg, volaarde2}).
\bdefi~
\begin{itemize}
\item A set $A\subset \R$ is \emph{dense} in $B\subset \R$ if for $k\in \N,b\in B$, there is $a\in A$ with $|a-b|<\frac{1}{2^{k}}$.
\item A function $f:\R\di \R$ is \emph{pointwise discontinuous} in case $C_{f}$ is dense in $\R$.
\item A set $A\subset \R$ is \emph{nowhere dense} \(in $ \R$\) if $A$ is not dense in any open interval.  
\end{itemize}
\edefi

Fifth, we need the `intermediate value property', also called `Darboux property'.  
\bdefi[Darboux property] Let $f:[0,1]\di \R$ be given. 
\begin{itemize}
\item A real $y\in \R$ is a left \(resp.\ right\) \emph{cluster value} of $f$ at $x\in [0,1]$ if there is $(x_{n})_{n\in \N}$ such that $y=\lim_{n\di \infty} f(x_{n})$ and $x=\lim_{n\di \infty}x_{n}$ and $(\forall n\in \N)(x_{n}\leq x)$ \(resp.\ $(\forall n\in \N)(x_{n}\geq x)$\).  
\item A point $x\in [0,1]$ is a \emph{Darboux point} of $f:[0,1]\di \R$ if for any $\delta>0$ and any left \(resp.\ right\) cluster value $y$ of $f$ at $x$ and $z\in \R$ strictly between $y$ and $f(x)$, there is $w\in (x-\delta, x)$ \(resp.\ $w\in ( x, x+\delta)$\) such that $f(w)=y$.   
\end{itemize}
\edefi
\noindent
By definition, a point of continuity is also a Darboux point, but not vice versa.

\section{Baire category theorem and computational equivalences}\label{main3}
In this section, we obtain the computational equivalences sketched in Section \ref{mintro} involving the Baire category theorem and basic properties of {Baire 1} and {pointwise discontinuous} functions (see Section \ref{cdef} for the latter).  
To avoid issues of the representation of real numbers, we will always assume $\mu^{2}$ or $\exists^{2}$ from Section \ref{lll}.

\subsection{Preliminaries}\label{prebaire}
As discussed in Section \ref{mintro}, we shall study Baire 1 functions that are given together with their \emph{oscillation function} (see Definition \ref{oscf} for the latter).  
We briefly discuss and motivate this construct in this section. 

\medskip

First of all, the study of regulated functions in \cites{dagsamXI, dagsamXII, dagsamXIII} is 
really only possible thanks to the associated left- and right limits (see Definition \ref{flung}) \emph{and} the fact that the latter are computable in $\exists^{2}$.  
Indeed, for regulated $f:\R\di \R$, the formula 
\be\label{figo}\tag{\textup{\textsf{C}}}
\text{\emph{ $f$ is continuous at a given real $x\in \R$}}
\ee
involves quantifiers over $\R$ but is equivalent to the \emph{arithmetical} formula $f(x+)=f(x)=f(x-)$.  
In this light, we can define the set $D_{f}$ of discontinuity points of $f$ -using only $\exists^{2}$- and proceed with the usual (textbook) proofs.  
%
Now, we would like to use an analogous approach for the study of pointwise discontinuous or Baire 1 functions.  
To this end, we consider the \emph{oscillation function} defined as follows.
\bdefi[Oscillation function]\label{oscf}
For any $f:\R\di \R$, the associated \emph{oscillation functions} are defined as follows: $\osc_{f}([a,b]):= \sup _{{x\in [a,b]}}f(x)-\inf _{{x\in [a,b]}}f(x)$ and $\osc_{f}(x):=\lim _{k \di \infty }\osc_{f}(B(x, \frac{1}{2^{k}}) ).$
\edefi
Riemann and Hankel already considered the notion of oscillation in the context of Riemann integration (\cites{hankelwoot, rieal}).  
Our main interest in Definition \ref{oscf} is that \eqref{figo} is now equivalent to the \emph{arithmetical} formula $\osc_{f}(x)=0$.  
Hence, in the presence of $\osc_{f}$, we can again define $D_{f}$ -using nothing more than $\exists^{2}$- and proceed with the usual (textbook) proofs. 
In the below, we will (only) study Baire 1 and pointwise discontinuous functions $f:[0,1]\di \R$ that are \emph{given together with} the associated oscillation function $\osc_{f}:[0,1]\di \R$.
We stress that such `constructive enrichment' is common in computability theory and constructive mathematics. 

\medskip

Secondly, we sketch the connection between Baire 1 functions and the Baire category theorem, in both directions.  
In one direction, fix a Baire 1 function $f:[0,1]\di \R$ and its oscillation function $\osc_{f}$.  A standard textbook technique is to decompose the set $D_{f}=\{ x\in [0,1]: \osc_{f}(x)>0  \}$ as the union of the closed sets
\be\label{kok}\textstyle
D_{k}:=\{ x\in [0,1]: \osc_{f}(x)\geq \frac{1}{2^{k}}  \} \textup{  for all $k\in \N$.}
\ee
The complement $O_{n}:= [0,1]\setminus D_{k}$ can be shown to be open and dense, as required for the antecedent of the Baire category theorem.  
This connection also goes in the other direction as follows: fix a sequence of dense and open sets $(O_{n})_{n\in \N}$ in the unit interval, define $X_{n}:= [0,1]\setminus O_{n}$ and consider 
the following function $h:[0,1]\di \R$:
\be\label{mopi}\tag{\textsf{\textup{H}}}
h(x):=
\begin{cases}
0 & x\not \in \cup_{m\in \N}X_{m} \\
\frac{1}{2^{n+1}} &  x\in X_{n} \textup{ and $n$ is the least such number}
\end{cases},
\ee
The function $h$ may be found in the literature and is Baire 1 (\cite{myerson}*{p.\ 238}).  
In conclusion, the Baire category theorem seems intimately connected to Baire 1 functions (in both directions), \emph{assuming} we have access to \eqref{kok}, which is why
we assume the oscillation $\osc_{f}$ to be given.    

\medskip
\noindent
Thirdly, we list basic facts about Baire 1 and pointwise discontinuous functions.
\begin{thm}\label{dorn}~
\begin{itemize}
\item Upper semi-continuous functions are Baire 1; the latter are pointwise discontinuous. 
\item Any $f:[0,1]\di \R$ is upper semi-continuous iff for any $a\in \R$, $f^{-1}([a, +\infty))$ is closed. 
\item For \emph{any} $\R\di \R$-function, the set $D_{f}$ is $\F_{\sigma}$, i.e.\ the union over $\N$ of closed sets.
\item For a sequence of closed $(X_{n})_{n\in \N}$, the function $h$ in \eqref{mopi} is Baire 1.   
\item Characteristic functions of subsets of the Cantor set (and variations) are pointwise discontinuous, but need not be Borel or Riemann integrable.
\item The class of bounded Baire 1 functions $\mathscr{B}_{1}$ is the union of all `small' Baire 1 classes $\mathscr{B}_{1}^{\xi}$ for $\xi<\omega_{1}$ \(see \cite{vuilekech} for the definition of the latter\).  
\end{itemize}
\end{thm}
\begin{proof}
Proofs may be found in e.g.\ \cite{oxi}*{\S7, p.\ 31-33} and \cites{beren2, myerson, vuilekech}.
\end{proof}
We will tacitly use Theorem \ref{dorn} in the below.  We will need one additional property of $h:[0,1]\di \R$ as in \eqref{mopi}, which we could not find in the literature. 
\begin{thm}\label{fronk}
Let $(X_{n})_{n\in \N}$ be a sequence of closed and nowhere dense sets and let $h:[0,1]\di \R$ be as in \eqref{mopi}.  
Then $\osc_{h}:[0,1]\di \R$ is computable in $\exists^{2}$. 
\end{thm}
\begin{proof}
Consider $h:[0,1]\di \R$ as in \eqref{mopi} where $(X_{n})_{n\in \N}$ is a sequence of closed nowhere dense sets.  
We will show that $h$ equals $\osc_{h}$ everywhere on $[0,1]$, i.e.\ $h$ is its own oscillation function. 
To this end, we proceed by the following case distinction. 
\begin{itemize}
\item In case $h(x_{0})=0$ for some $x_{0}\in [0,1]$, then $x_{0}\in \cap_{n\in \N}Y_{n}$ where $Y_{n}:= [0,1]\setminus X_{n}$ is open.  Hence, for any $m\in \N$, there is $N\in \N$ such that $B(x_{0}, \frac{1}{2^{N}})\subset \cap_{n\leq m}Y_{n}$, as the latter intersection is open. By the definition of $\osc_{h}$, we have $\osc_{h}(x_{0})<\frac{1}{2^{m}}$ for all $m\in \N$, i.e.\ $\osc_{h}(x_{0})=h(x_{0})=0$.  
\item In case $\osc_{h}(x_{0})=0$ for some $x_{0}\in [0,1]$, we must have $x_{0}\not \in  \cup_{n\in \N}X_{n}$ and hence $h(x_{0})=0$ by definition.  Indeed, if $x_{0}\in X_{n_{0}}$, then $\osc_{h}(x_{0})\geq \frac{1}{2^{n_{0}}}$ because $\inf_{x\in B(x_{0}, \frac{1}{2^{k}})}h(x)=0$ (for any $k\in \N$) due to $\cap_{n\in \N}Y_{n}$ being dense in $[0,1]$, while of course $\sup_{x\in B(x_{0}, \frac{1}{2^{k}})}h(x)\geq h(x_{0})\geq \frac{1}{2^{n_{0}}}$.   
\item In case $h(x_{0})=\frac{1}{2^{n_{0}+1}}$ for some $x_{0}\in [0,1]$, suppose $\osc_{h}(x_{0})\ne \frac{1}{2^{n_{0}+1}}$.   Since by definition (and the previous item) $\osc_{h}(x_{0})\geq h(x_{0})=\frac{1}{2^{n_{0}+1}}$, we have $\osc_{h}(x_{0})>\frac{1}{2^{n_{0}+1}}$, implying $\osc_{h}(x_{0})\geq\frac{1}{2^{n_{0}}}$ and $n_{0}>0$.  Now, if $x_{0}\in O:= \cap_{n\leq n_{0-1}}Y_{n}$, then $B(x_{0}, \frac{1}{2^{N}})\subset O$ for $N$ large enough, as $O$ is open; by definition, the latter inclusion implies that $\osc_{h}(x_{0})\leq \frac{1}{2^{n_{0}+1}}$, a contradiction.  However, $x_{0}\not \in O$  (which is equivalent to $x_{0}\in  \cup_{n\leq n_{0}-1}X_{n}$), also leads to a contradiction as then $h(x_{0})>\frac{1}{2^{n_{0}+1}}$.  In conclusion, we have $h(x_{0})=\frac{1}{2^{n_{0}+1}}=\osc_{h}(x_{0})$. 
\item In case $\osc_{h}(x_{0})>0$ for some $x_{0}\in [0,1]$, suppose $h(x_{0})\ne \osc_{h}(x_{0})$.   
By definition (and the first item) we have $\osc_{h}(x_{0})\geq h(x_{0})>0$, implying $\osc_{h}(x_{0})>h(x_{0})=\frac{1}{2^{n_{0}+1}}$ for some $n_{0}\in \N$.  
In turn, we must have $\osc_{h}(x_{0})\geq \frac{1}{2^{n_{0}}}$ and $n_{0}>0$.
Now, if $x_{0}\in O:= \cap_{n\leq n_{0}-1}Y_{n}$, then $B(x_{0}, \frac{1}{2^{N}})\subset O$ for $N$ large enough, as $O$ is open; by definition, the latter inclusion implies that $\osc_{h}(x)\leq \frac{1}{2^{n_{0}+1}}$, a contradiction.  However, $x_{0}\not \in O$  (which is equivalent to $x_{0}\in  \cup_{n\leq n_{0}-1}X_{n}$), also leads to a contradiction as then $h(x_{0})\geq\frac{1}{2^{n_{0}}}$.  
Since both cases lead to contradiction, we have $h(x_{0})=\osc_{h}(x_{0})$. 
\end{itemize}
In conclusion, we have $h(x)=\osc_{h}(x)$ for all $x\in [0,1]$, as required.  
\end{proof}
The previous theorem is perhaps surprising: computing the oscillation function of arbitrary functions readily yields $\exists^{3}$, while the function $h$ from \eqref{mopi} comes `readily equipped' with $\osc_{h}$.

\subsection{Main results}
In this section, we establish the computational equivalences sketched in Section \ref{mintro}.  The function $h$ from Section \ref{prebaire} and Theorem \ref{fronk} play a central role. 

\medskip

In particular, we have the following theorem where we note that Volterra's original results from \cite{volaarde2} (see Section \ref{vintro}) are formulated for pointwise discontinuous functions.  
We also note that e.g.\ Dirichlet's function is in Baire class 2, i.e.\ we cannot go higher in the Baire hierarchy (without extra effort). 
\begin{thm}\label{nolabel}
Assuming $\exists^{2}$, the following are computationally equivalent.
\begin{enumerate}
 \renewcommand{\theenumi}{\alph{enumi}}
\item A Baire realiser, i.e.\ a functional that on input a sequence $(O_{n})_{n\in \N}$ of dense and open subsets of $[0,1]$, outputs $x\in \cap_{n\in \N}O_{n}$.\label{benga1}
\item A functional that on input a Baire 1 function $f:[0,1]\di \R$ and its oscillation $\osc_{f}:[0,1]\di \R$, outputs $y\in [0,1]$ where $f$ is continuous \(or quasi-continuous, or lower semi-continuous, or Darboux\).\label{benga2}
\item \(Volterra\) A functional that on input a Baire 1 function $f:[0,1]\di \R$ and its oscillation $\osc_{f}:[0,1]\di \R$, outputs either $q\in \Q\cap [0,1]$ where $f$ is discontinuous, or $x\in [0,1]\setminus \Q$ where $f$ is continuous. \label{benga3}
\item \(Volterra\) A functional that on input Baire 1 functions $f,g:[0,1]\di \R$ and their oscillation functions $\osc_{f}, \osc_{g}:[0,1]\di \R$, outputs a real $x\in [0,1]$ such that $f$ and $g$ are both continuous or both discontinuous at $x$. \label{benga4}
\item \(Baire, \cite{beren2}*{p.\ 66}\) A functional that on input a sequence of Baire 1 functions $(f_{n})_{n\in \N}$ and their oscillation functions $(\osc_{f_{n}})_{n\in \N}$, outputs a real $x\in [0,1]$ such that all $f_{n}$ are continuous at $x$. \label{benga5}
\item A functional that on input Baire 1 $f:[0,1]\di \R$ and its oscillation function $\osc_{f}$, outputs $a, b\in  [0,1]$ such that $\{ x\in [0,1]:f(a)\leq f(x)\leq f(b)\}$ is infinite.\label{benga6}
\item The previous items with `Baire 1' generalised to `pointwise discontinuous'.\label{bengafinal}
\item The previous items with `Baire 1' restricted to `upper semi-continuous'. \label{bengafinal3}
\item The previous items with `Baire 1' restricted to `small Baire class $\mathscr{B}_{1}^{\xi}$' for any countable ordinal $1\leq \xi<\omega_{1}$. \label{bengafinal4}
\end{enumerate}
\end{thm}
\begin{proof}
First of all, many results will be proved using $h:[0,1]\di \R$ as in \eqref{mopi}.  
By Theorem~\ref{fronk}, the associated oscillation function $\osc_{h}:[0,1]\di \R$ is available, which we will tacitly assume.

\medskip

For the implication \eqref{benga1} $\di$ \eqref{benga2}, let $f:[0,1]\di \R$ be Baire 1 and let $\osc_{f}:[0,1]\di \R$ be its oscillation.  
The following set, readily defined using $\exists^{2}$, is closed and nowhere dense, as can be found in e.g.\ \cite{oxi}*{p.\ 31, \S7}:
\be\label{tachyon2}  \textstyle
D_{k}:=\{ x\in [0,1] : \osc_{f}(x)\geq \frac{1}{2^{k}} \},
\ee
The union $D_{f}=\cup_{k\in \N}D_{k}$ collects all points where $f$ is discontinuous.  Hence, $O_{k}:= [0,1]\setminus D_{k}$ is open and dense for all $k\in \N$, while $y\in \cap_{k\in\N}O_{k}$ implies $\osc_{f}(y)=0$, i.e.\ $y$ is a point of continuity of $f$ by definition (of the oscillation function).  Hence, we have established that a Baire realiser computes some $x\in C_{f}$ for a Baire 1 function $f$ and its oscillation function, i.e.\ item \eqref{benga2} follows.   
The generalisation involving pointwise discontinuous functions (item \eqref{bengafinal}) is now immediate, as the very same proof goes through. 

\medskip

For the implication \eqref{benga2} $\di $ \eqref{benga1}, let $(O_{n})_{n\in \N}$ be a sequence of open and dense sets.  
Then $(X_{n})_{n\in \N}$ for $X_{n}:= [0,1]\setminus O_{n}$ is a sequence of closed and nowhere dense sets. 
Now consider the function $h:[0,1]\di \R$ from \eqref{mopi}, which is Baire 1 as noted above.  
Item \eqref{benga2} yields $y\in [0,1]$ where $h$ is continuous and we must have $y \in \cap_{n\in \N}O_{n}$.  
Indeed, in case $y \in X_{m_{0}}$, there is $N\in \N$ such that $h(z)>\frac{1}{2^{m_{0}+1}}$ for $z\in B(y, \frac{1}{2^{N}})$, by the continuity of $h$ at $y$. 
However, since $O:=\cap_{n\leq m_{0}+2}O_{n}$ is dense in $[0,1]$, there is some $z_{0}\in B(y, \frac{1}{2^{N}}) \cap O$, which yields $h(z_{0})\leq \frac{1}{2^{m_{0}+2}}$ by the definition of $h$, a contradiction. 
Hence, we have $y \in \cap_{n\in \N}O_{n}$ as required by the specification of Baire realisers as in item~\eqref{benga1}.
The same argument works for quasi-continuity, lower semi-continuity, and the (local) Darboux property. 

\medskip

For the implication \eqref{benga2} $\di $ \eqref{benga3}, let $f$ and $\osc_{f}$ be as in the latter item.  
Note that $\mu^{2}$ can find $q\in \Q\cap [0,1]$ such that $\osc_{f}(q)>_{\R}0$, if such exists.  
If there is no such rational, the functional from item \eqref{benga2} must output an \emph{irrational} $y\in [0,1]$ such that $f$ is continuous at $y$. 
Hence, item \eqref{benga3} follows while the generalisation involving pointwise discontinuous functions (item \eqref{bengafinal}) is again immediate. 

\medskip

For the implication \eqref{benga3} $\di $ \eqref{benga2}, let $f$ and $\osc_{f}$ be as in the latter item.  
Note that $\mu^{2}$ can find $q\in \Q\cap [0,1]$ such that $\osc_{f}(q)=_{\R}0$, if such exists; such rational is a point of continuity of $f$, which is are required for item \eqref{benga2}.  
If there is no such rational, the functional from item \eqref{benga3} must output an {irrational} $y\in [0,1]$ such that $f$ is continuous at $y$. 
Hence, item \eqref{benga2} follows from item \eqref{benga3} while the generalisation involving pointwise discontinuous functions (item \eqref{bengafinal}) is again immediate. 

\medskip

Next, assume item \eqref{benga4} and recall Thomae's function $T$ as in \eqref{thomae}, which is regulated and hence Baire 1; the oscillation function is $\osc_{T}$ readily defined using $\exists^{2}$.  
Since $T$ is continuous on $\R\setminus \Q$, item \eqref{benga4} yields item \eqref{benga3}.  To obtain item \eqref{benga4} from a Baire realiser, let $f, g:[0,1]$ be Baire 1 with oscillation functions $\osc_{f}$ and $\osc_{g}$.  Consider the set $D_{k}$ as in \eqref{tachyon2} and let $E_{k}$ be the same set for $g$, i.e.\ defined as: 
\be\label{tachyon3}  \textstyle
E_{k}:=\{ x\in [0,1] : \osc_{g}(x)\geq \frac{1}{2^{k}} \}.
\ee
Then $O_{n}:=[0,1]\setminus (D_{n}\cup E_{n})$ is open and dense as in the previous paragraphs. 
Now use a Baire realiser to obtain $y\in \cap_{n\in \N}O_{n}$.  By definition, $f$ and $g$ are continuous at $y$, i.e.\ item \eqref{benga2} follows.
The generalisation for pointwise discontinuous functions (item \eqref{bengafinal}) is immediate. 

\medskip
\noindent
For item \eqref{benga5}, consider the following (nowhere dense and closed as for \eqref{tachyon2}) set:
\be\label{tachyon}  \textstyle
D_{k, n}:=\{ x\in [0,1] : \osc_{f_{n}}(x)\geq \frac{1}{2^{k}} \} \textup{   for $k,n\in \N$,}
\ee
and where each $f_{n}:[0,1]\di \R$ is Baire 1 with oscillation $\osc_{f_{n}}:[0,1]\di \R$.
Since each $f_{n}$ is continuous outside of $D_{f_{n}}:=\cup_{k\in \N}D_{k,n}$, any $y\not \in \cup_{k, n\in \N}D_{k, n}$ is such that each $f_{n}$ is continuous at $y$.
Hence, item \eqref{benga5} follows from item \eqref{benga1}.  

\medskip

For item \eqref{benga6}, consider $h$ as in \eqref{mopi} and note that if $h(a)>0$, then $\{ x\in [0,1]:h(a)\leq h(x)\leq h(b)\}$ is finite.  
Hence, item \eqref{benga6} must provide $a\in [0,1]$ such that $h(a)=0$, which by definition satisfies $a\not\in \cup_{n\in \N}X_{n}$, i.e.\ item \eqref{benga1} follows. 
%
%

\medskip

%
%
%
Finally, regarding item \eqref{bengafinal3}, the function $h:[0,1]\di \R$ from \eqref{mopi} is in fact upper semi-continuous, assuming $(X_{n})_{n\in \N}$ is a sequence of closed and nowhere dense sets.  
One can shows this directly using the definition of semi-continuity, or observe that $h^{-1}([a, +\infty))$ is either $\emptyset$, $[0,1]$, or a finite union of $X_{i}$, all of which are closed.  
Regarding item \eqref{bengafinal4}, the class $\mathscr{B}_{1}^{\xi}$ includes the (upper and lower) semi-continuous functions if $\xi\geq 1$ (\cite{vuilekech}*{\S2}).  Hence, the equivalences for the restrictions of item \eqref{benga2} as in items \eqref{bengafinal3} or \eqref{bengafinal4} have been established; the very same argument applies to the other items, i.e.\ items \eqref{bengafinal3} and \eqref{bengafinal4} are finished. 
%
%
%
%
%
%
%
\end{proof}
The reader easily verifies that most equivalences in Theorem \ref{nolabel} go through in a small fragment of G\"odel's $T$.

\medskip

Next, we isolate the following result as we need to point out the meaning of `countable set' as an input of a functional.
As in \cite{dagsamXIII, dagsamXII, samwollic22}, we assume that a countable set $D\subset \R$ is given together with $Y:[0,1]\di \N$ which is injective on $D$ \textbf{and} that 
these two objects $D$ and $Y$ are an input for the functional at hand, e.g.\ as in item \eqref{benga7} from Corollary \ref{chron}. 
\begin{cor}\label{chron}
Given $\exists^{2}$, the following is computationally equivalent to a Baire realiser:
\begin{enumerate}
 \renewcommand{\theenumi}{\alph{enumi}}
\setcounter{enumi}{10}
\item A functional that on input a countable dense $D\subset [0,1]$, a function $f:[0,1]\di \R$ in Baire 1, and its oscillation $\osc_{f}$, 
outputs either $d\in D$ such that $f$ is discontinuous at $d$, or $x\not\in D$ such that $f$ is continuous at $x$.\label{benga7}
\end{enumerate}
The equivalence remains correct if we require a bijection from $D$ to $\N$ \(instead of an injection\). 
\end{cor}
\begin{proof}
For item \eqref{benga7}, the latter clearly implies item \eqref{benga3} for $D=\Q$, even if we require a bijection from $D$ to $\N$. 
To establish item \eqref{benga7} assuming item \eqref{benga1}, let $D\subset [0,1]$ and $Y:[0,1]\di \N$ be such that the former is dense and the latter is injective on the former.  
Now fix $f:[0,1]\di \R$ in Baire 1 and $\osc_{f}$.  Define $E_{k}:= D_{k}\cup \{x\in D: Y(x)\leq k\}$ where the former is as in \eqref{tachyon2} and the latter is finite.  
Hence, $O_{n}:=[0,1]\setminus E_{n}$ is open and dense.   Any Baire realiser therefore provides $y\in \cap_{n\in \N}O_{n}$, which is such that $f$ is continuous at $y\not \in D$, as required.  
\end{proof}
Note that item \eqref{benga7} in the corollary is based on the generalisation of Volterra's theorem as in Theorem \ref{dorki}.  
We believe many similar results can be established, e.g.\ by choosing different (but equivalent over strong systems) definitions of `countable set', as discussed in Section \ref{froli}.

\medskip

Next, we discuss potential generalisations of Theorem \ref{nolabel}.  
\begin{rem}[Generalisations]\rm
First of all, items \eqref{bengafinal3} and \eqref{bengafinal4} of Theorem \ref{nolabel} imply that certain `small' sub-classes of $\mathscr{B}_{1}$ already yield the required equivalences.
One may wonder whether the same holds for the sub-classes $\mathscr{B}_{1}^{*}$ and $\mathscr{B}_{1}^{**}$ of $\mathscr{B}_{1}$ (see e.g.\ \cite{notsieg} for the former).  As far as we can see, 
the function $h$ (and variations) from \eqref{mopi} does not belong to $\mathscr{B}_{1}^{*}$ or $\mathscr{B}_{1}^{**}$.

\medskip

Secondly, pointwise discontinuous functions are \emph{exactly} those for which $D_{f}$ is \emph{meagre} (\cites{beren2, myerson, vuilekech, oxi}), i.e. the union of nowhere dense sets.  
Since the complement of closed and nowhere dense sets is open and dense, it seems we cannot generalise Theorem \ref{nolabel} beyond pointwise discontinuous functions (without extra effort). 
\end{rem}

\section{Related results}\label{related}
We discuss some results related (directly or indirectly) to Baire realisers. 
\subsection{Baire 1 functions and open sets}\label{dichtbij}
In this section, we motivate the study of the definition of open set from \cite{dagsamVII} and Definition \ref{char} in Section \ref{cdef}, which amounts to \eqref{R1} below.  
In particular, Theorem \ref{bootheel} below expresses that the closed sets $D_{k}$ as in \eqref{tachyon2} are just closed sets 
\emph{without any additional information}.  Hence, the study of Baire 1 already boasts (general) open sets that only satisfy the \eqref{R1}-definition of open sets, i.e.\ for which other representations (see \eqref{R2}-\eqref{R4} below) are not computable, even assuming $\SS_{k}^{2}$; we conjecture that we can add non-trivial non-normal functionals. 

\medskip
\noindent
First of all, the following representations of open sets were studied in \cite{dagsamVII}*{\S7}.
\bdefi[Representations of open sets]~
\begin{enumerate}
\renewcommand{\theenumi}{R.\arabic{enumi}}
\item The open set $O$ is represented by its characteristic function.  
We just have the extra information that $O$ is open \(Definition \ref{char}\).\label{R1}
\item The set $O$ is represented by a function $Y : [0,1] \rightarrow [0,1]$ such that
\begin{itemize}
\item[(i)] we have $O = \{x \mid Y(x) >_{\R} 0\}$, 
\item[(ii)] if $Y(x) > 0$ , then $(x-Y(x),x+Y(x))\cap [0,1] \subseteq O$. 
\end{itemize}\label{R2}
\item The set $O$ is represented by the \emph{continuous} function $Y$ where
\begin{itemize} 
\item[(i)] $Y(x)$ is the distance from $x$ to $[0,1]\setminus O$ if the latter is nonempty,
\item[(ii)] $Y$ is constant 1 if $O = [0,1]$.
\end{itemize}\label{R3}
\item The set $O$ is given as the union of a sequence of open rational intervals $(a_i,b_i)$, the sequence being a representation of $O$.\label{R4}
\end{enumerate}
\edefi
We note that the \eqref{R4}-representation is used in frameworks based on second-order arithmetic (see e.g.\ \cite{simpson2}*{II.5.6}).

\medskip

Secondly, assuming $\exists^2$, it is clear that the information given by a representation increases when going down the list.   
As expected, \eqref{R3} and \eqref{R4} are `the same' by the following theorem.
\begin{thm}\label{friuk}
The functional $\exists^{2}$ computes an \eqref{R3}-representation given an \eqref{R4}-representation, and vice versa. 
\end{thm}
\begin{proof}
Immediate from \cite{dagsamVII}*{Theorem 7.1}.
\end{proof}
The \eqref{R2} and \eqref{R3} representations are not computationally equivalent.  
Indeed, the functional $\Delta^{3}$ from \cite{dagsamVII}*{\S7.1} converts open sets as in \eqref{R2} to open sets as in \eqref{R3}.
This functional is not computable in any $\SS_{k}^{2}$ and computes nothing new when combined with $\exists^{2}$, as discussed in detail in \cite{dagsamVII}*{\S7}. 

\medskip

Thirdly, we have the following main theorem, where $\Omega_{\fin}$ is studied in \cite{dagsamXI, dagsamXIII,samwollic22} based on the following definition.
\bdefi
A \emph{finiteness realiser} $\Omega_{\fin}$ is defined when the input $X\subset \R$ is finite and $\Omega_{\fin}(X)$ then outputs a finite sequence of reals that includes all elements of $X$.
\edefi
We note that $\Omega_{\fin}$ is explosive in that $\Omega_{\fin}+\SS_{1}^{2}$ computes $\SS_{2}^{2}$ (see \cite{dagsamXI, dagsamXIII}).
\begin{thm}\label{bootheel}
Together with $\exists^{2}$, the following computes a Baire realiser \(Def.~\ref{pip}\):
\begin{center}
a functional that on input $f:[0,1]\di \R$ in Baire 1, its oscillation function $\osc_{f}$, and $n\in \N$, outputs an \eqref{R2}-representation of $[0,1]\setminus D_{n}$ as in \eqref{tachyon2}. 
\end{center}
Combining the centred operation with $\Delta$, which converts \eqref{R2}-representations into \eqref{R3}-ones, we can compute $\Omega_{\fin}$.
\end{thm}
\begin{proof}
For the first part, the centred operation provides an \eqref{R2}-representation for the open and dense set $O_{n}:=[0,1]\setminus D_{n}$.
It is a straightforward verification that, given an \eqref{R2}-representation for a sequence of open and sense sets, the usual constructive proof of the Baire category theorem goes through (see \cite{bish1}*{p.\ 84} for the latter).  
A detailed proof, assuming $\exists^{2}$, can be found in \cite{dagsamVII}*{Theorem 7.10}.  In other words, given an \eqref{R2}-representation for each $O_{n}$, we can compute $y\in \cap_{n\in \N}O_{n}$ assuming $\exists^{2}$, and we are done.

\medskip

For the second part, let $X\subset \R$ be finite.  Define $F_{X}$ as $1$ in case $x\in X$, and $0$ otherwise.  Clearly, $F_{X}$ is Baire 1 and the function $\osc_{F_{X}}$ is just $F_{X}$ itself. 
The centred operation form the theorem, together with $\Delta$ and Theorem \ref{friuk}, provides an \eqref{R4} representation of each $[0,1] \setminus D_{n}$. 
However, by \cite{dagsamXIII}*{Lemma 4.11}, there is a functional $\mathcal{E}$, computable in $\exists^2$, such that for any sequences $(a_{n})_{n\in \N}$, $(b_{n})_{n\in \N}$, if the closed set $C=[0,1]\setminus \cup_{n\in \N}(a_{n}, b_{n}) $ is countable, then $\mathcal{E}(\lambda n.(a_{n}, b_{n}))$ enumerates the points in $C$.  Hence, functional $\mathcal{E}$ can enumerate the finite set $D_{k}$ for any $k\in\N$, and therefore yields an enumeration of $X$.
\end{proof}
Note that for the second part, we could restrict the centred operation to some fixed $n\in \N$.

\subsection{Functionals related to Baire realisers}\label{froli}
We show that Baire realisers compute certain witnessing functionals for the uncountability of $\R$, called \emph{strong Cantor realisers} in \cite{samwollic22}.

\medskip

First of all, we introduce the following notion of countable set, equivalent to the usual definition over strong enough logical systems.

\begin{defi}[Height countable]\label{hoogzalieleven}
A set $A\subset \R$ is \emph{height countable} if there is a \emph{height} $H:\R\di \N$ for $A$, i.e.\ for all $n\in \N$, $A_{n}:= \{ x\in A: H(x)<n\}$ is finite.  
\end{defi}
The notion of `height' is mentioned in e.g.\ \cite{demol}*{p.\ 33} and \cite{vadsiger, royco}, from whence we took this notion.   The following remark is crucial for the below.
\begin{rem}[Inputs and heights]\label{inhe}\rm
A functional defined on height countable sets always takes \textbf{two} inputs: the set $A\subset \R$ \textbf{and} the height $H:\R\di \N$.  
We note that given a sequence of finite sets $(X_{n})_{n\in \N}$ in $\R$, a height can be defined as $H(x):= (\mu n)(x\in X_{n})$ using Feferman's $\mu^{2}$. 
A set is therefore height countable iff it is the union over $\N$ of finite sets.  
\end{rem}
The following witnessing functional for the uncountability of $\R$ was introduced in \cite{samwollic22}. 
\begin{defi}[Strong Cantor realiser]
A \emph{strong Cantor realiser} is any functional that on input a height countable $A\subset [0,1]$, outputs $y\in [0,1]\setminus A$. 
\end{defi}
Following Remark \ref{inhe}, a strong Cantor realiser takes as input a countable set \textbf{and} its height as in Definition \ref{hoogzalieleven}.
Modulo $\exists^{2}$, this amounts to an input consisting of a sequence $(X_{n})_{n\in \N}$ of finite sets in $[0,1]$ and an output $y\in [0,1]\setminus \cup_{n\in \N}X_{n}$.
It was shown in \cite{samwollic22} that many operations on regulated functions are computationally equivalent to strong Cantor realisers, thereby justifying Definition \ref{hoogzalieleven}.  
By contrast, we could not find comparable computational equivalences for `normal' and `weak' Cantor realisers (see \cite{dagsamXII, samwollic22}), which make use of the more standard definition of `countable set' based on injections and bijections to $\N$.

\medskip

Next, we have the following theorem connecting some of the aforementioned functionals.
\begin{thm} Assuming $\exists^{2}$, we have that
\begin{itemize}
\item any Baire realiser computes a strong Cantor realiser,
\item the functional $\Omega_{\fin}$ computes a strong Cantor realiser.
\end{itemize}
\end{thm}
\begin{proof}
Fix a height countable set $A\subset [0,1]$ and its height function $H:\R\di \N$.  
By definition $A_{n}:= \{x\in [0,1]: H(x)<n\}$ is finite, hence $O_{n}:= [0,1]\setminus A_{n}$ is open and dense.  
A Baire realiser provides $y\in \cap_{n\in \N}O_{n}$, which immediately yields $y\not \in A$.  
Similarly, $\Omega_{\fin}(A_{n})$ yields an enumeration of $A_{n}$ and $\mu^{2}$ readily yields a sequence $(a_{n})_{n\in \N}$ containing all elements of $A$. 
There are (efficient) algorithms (see \cite{grayk}) to find $y\in [0,1]$ such that $y\ne a_{n}$ for all $n\in \N$, and we are done.
\end{proof}
On a related note, the reader easily verifies that Corollary \ref{chron} goes through for `countable' replaced by `height-countable', keeping in mind Remark \ref{inhe}.
We believe that many such variations are possible.

\appendix

\section{Some details of Kleene's computability theory}\label{appendisch}
Kleene's computability theory borrows heavily from type theory and higher-order arithmetic.  We briefly sketch some of the associated notions.  

\medskip

First of all, Kleene's S1-S9 schemes define `$\Psi$ is computable in terms of $\Phi$' for objects $\Phi,\Psi$ of any finite type.  
Now, the collection of \emph{all finite types} $\mathbf{T}$ is defined by the following two clauses:
\begin{center}
(i) $0\in \mathbf{T}$   and   (ii)  If $\sigma, \tau\in \mathbf{T}$ then $( \sigma \di \tau) \in \mathbf{T}$,
\end{center}
where $0$ is the type of natural numbers, and $\sigma\di \tau$ is the type of mappings from objects of type $\sigma$ to objects of type $\tau$.
In this way, $1\equiv 0\di 0$ is the type of functions from numbers to numbers, and  $n+1\equiv n\di 0$.  
We view sets $X$ of type $\sigma$ objects as given by characteristic functions $F_{X}^{\sigma\di 0}$.  

\medskip

Secondly, for variables $x^{\rho}, y^{\rho}, z^{\rho},\dots$ of any finite type $\rho\in \mathbf{T}$,   types may be omitted when they can be inferred from context.  
The constants include the type $0$ objects $0, 1$ and $ <_{0}, +_{0}, \times_{0},=_{0}$  which are intended to have their usual meaning as operations on $\N$.
Equality at higher types is defined in terms of `$=_{0}$' as follows: for any objects $x^{\tau}, y^{\tau}$, we have
\be\label{aparth}
[x=_{\tau}y] \equiv (\forall z_{1}^{\tau_{1}}\dots z_{k}^{\tau_{k}})[xz_{1}\dots z_{k}=_{0}yz_{1}\dots z_{k}],
\ee
if the type $\tau$ is composed as $\tau\equiv(\tau_{1}\di \dots\di \tau_{k}\di 0)$.  

\medskip

Thirdly, we introduce the usual notations for common mathematical notions, like real numbers, as also can be found in \cite{kohlenbach2}.  
\begin{defi}[Real numbers and related notions]\label{keepintireal}\rm~
\begin{enumerate}
 \renewcommand{\theenumi}{\alph{enumi}}
\item Natural numbers correspond to type zero objects, and we use `$n^{0}$' and `$n\in \N$' interchangeably.  Rational numbers are defined as signed quotients of natural numbers, and `$q\in \Q$' and `$<_{\Q}$' have their usual meaning.    
\item Real numbers are coded by fast-converging Cauchy sequences $q_{(\cdot)}:\N\di \Q$, i.e.\  such that $(\forall n^{0}, i^{0})(|q_{n}-q_{n+i}|<_{\Q} \frac{1}{2^{n}})$.  
We use Kohlenbach's `hat function' from \cite{kohlenbach2}*{p.\ 289} to guarantee that every $q^{1}$ defines a real number.  
\item We write `$x\in \R$' to express that $x^{1}:=(q^{1}_{(\cdot)})$ represents a real as in the previous item and write $[x](k):=q_{k}$ for the $k$-th approximation of $x$.    
\item Two reals $x, y$ represented by $q_{(\cdot)}$ and $r_{(\cdot)}$ are \emph{equal}, denoted $x=_{\R}y$, if $(\forall n^{0})(|q_{n}-r_{n}|\leq {2^{-n+1}})$. Inequality `$<_{\R}$' is defined similarly.  
We sometimes omit the subscript `$\R$' if it is clear from context.           
\item Functions $F:\R\di \R$ are represented by $\Phi^{1\di 1}$ mapping equal reals to equal reals, i.e.\ extensionality as in $(\forall x , y\in \R)(x=_{\R}y\di \Phi(x)=_{\R}\Phi(y))$.\label{EXTEN}
\item The relation `$x\leq_{\tau}y$' is defined as in \eqref{aparth} but with `$\leq_{0}$' instead of `$=_{0}$'.  Binary sequences are denoted `$f^{1}, g^{1}\leq_{1}1$', but also `$f,g\in C$' or `$f, g\in 2^{\N}$'.  Elements of Baire space are given by $f^{1}, g^{1}$, but also denoted `$f, g\in \N^{\N}$'.
\item Sets of type $\rho$ objects $X^{\rho\di 0}, Y^{\rho\di 0}, \dots$ are given by their characteristic functions $F^{\rho\di 0}_{X}\leq_{\rho\di 0}1$, i.e.\ we write `$x\in X$' for $ F_{X}(x)=_{0}1$. \label{koer} 
\end{enumerate}
\end{defi}
For completeness, we list the following notational convention for finite sequences.  
\begin{nota}[Finite sequences]\label{skim}\rm
The type for `finite sequences of objects of type $\rho$' is denoted $\rho^{*}$, which we shall only use for $\rho=0,1$.  
We shall not always distinguish between $0$ and $0^{*}$. 
Similarly, we assume a fixed coding for finite sequences of type $1$ and shall make use of the type `$1^{*}$'.  
In general, we do not always distinguish between `$s^{\rho}$' and `$\langle s^{\rho}\rangle$', where the former is `the object $s$ of type $\rho$', and the latter is `the sequence of type $\rho^{*}$ with only element $s^{\rho}$'.  The empty sequence for the type $\rho^{*}$ is denoted by `$\langle \rangle_{\rho}$', usually with the typing omitted.  

\medskip

Furthermore, we denote by `$|s|=n$' the length of the finite sequence $s^{\rho^{*}}=\langle s_{0}^{\rho},s_{1}^{\rho},\dots,s_{n-1}^{\rho}\rangle$, where $|\langle\rangle|=0$, i.e.\ the empty sequence has length zero.
For sequences $s^{\rho^{*}}, t^{\rho^{*}}$, we denote by `$s*t$' the concatenation of $s$ and $t$, i.e.\ $(s*t)(i)=s(i)$ for $i<|s|$ and $(s*t)(j)=t(|s|-j)$ for $|s|\leq j< |s|+|t|$. For a sequence $s^{\rho^{*}}$, we define $\overline{s}N:=\langle s(0), s(1), \dots,  s(N-1)\rangle $ for $N^{0}<|s|$.  
For a sequence $\alpha^{0\di \rho}$, we also write $\overline{\alpha}N=\langle \alpha(0), \alpha(1),\dots, \alpha(N-1)\rangle$ for \emph{any} $N^{0}$.  By way of shorthand, 
$(\forall q^{\rho}\in Q^{\rho^{*}})A(q)$ abbreviates $(\forall i^{0}<|Q|)A(Q(i))$, which is (equivalent to) quantifier-free if $A$ is.   
\end{nota}

\end{document}